\documentclass{ajm}
\usepackage{palatino, mathpazo}
\usepackage[mathscr]{eucal}
\usepackage{amsmath,amsthm,amsfonts,amssymb, mathrsfs, mathtools}
\usepackage[all]{xy}
\usepackage{hyperref}

\newtheorem{theorem}{Theorem}[section]
\newtheorem{lemma}[theorem]{Lemma}
\newtheorem{proposition}[theorem]{Proposition}

\newtheorem{corollary}[theorem]{Corollary}
\theoremstyle{remark}
\newtheorem{remark}[theorem]{Remark}
\newtheorem{example}[theorem]{Example}
\newtheorem{definition}[theorem]{Definition}
\newtheorem{question}{Question}
\numberwithin{equation}{section}

\newcommand{\hiota}{\widehat{\iota}}

\newcommand{\x}{\widehat{X}}

\newcommand{\tX}{\widetilde{X}}

\newcommand{\hsC}{\widehat{\mathscr{C}}}
\newcommand{\hsL}{\widehat{\mathscr{L}}}
\newcommand{\hU}{\widehat{U}}

\newcommand{\fm}{\mathfrak{m}}

\newcommand{\bZ}{\mathbb{Z}}
\newcommand{\bQ}{\mathbb{Q}}
\newcommand{\bC}{\mathbb{C}}
\newcommand{\bP}{\mathbb{P}}

\newcommand{\chX}{\widehat{\mathcal{X}}}
\newcommand{\cX}{\mathcal{X}}
\newcommand{\cV}{\mathcal{V}}

\newcommand{\sC}{\mathscr{C}}

\newcommand{\sO}{\mathscr{O}}
\newcommand{\sL}{\mathscr{L}}
\newcommand{\sE}{\mathscr{E}}
\newcommand{\sF}{\mathscr{F}}

\newcommand{\Exc}{\mathrm{Exc}}
\newcommand{\Sing}{\mathrm{Sing}}
\newcommand{\sing}{\mathrm{Sing}}
\newcommand{\Def}{\mathrm{Def}}
\newcommand{\Spec}{\mathrm{Spec}\,}
\newcommand{\Bl}{\mathrm{Bl}}
\newcommand{\rk}{\mathrm{rank}}

\renewcommand{\labelenumi}{(\theenumi)}

\title[The connectedness of the standard web of Calabi--Yau 3-folds]{On the connectedness of the standard web of Calabi-Yau 3-folds and small transitions}

\author[S.-S.\ Wang]{Sz-Sheng Wang}
\address{S.-S.\ Wang: Department of Mathematics, National Taiwan University, Taipei, Taiwan}
\email{d98221004@ntu.edu.tw}

\thanks{Supported by MOST project 103-2115-M-002-002-MY3.}

\begin{document}
\maketitle

\begin{abstract}
We supply a detailed proof of the result by P.S. Green and T. H$\ddot{\text{u}}$bsch that all complete intersection Calabi--Yau 3-folds in product of projective spaces are connected through projective conifold transitions (known as the standard web). We also introduce a subclass of small transitions which we call \emph{primitive} small transitions and study such subclass. More precisely, given a small projective resolution $\pi : \widehat{X} \rightarrow X$ of a Calabi--Yau 3-fold $X$, we show that if the natural closed immersion $\Def(\widehat{X}) \hookrightarrow \Def(X)$ is an isomorphism then $X$ has only ODPs as singularities.
\end{abstract}

\section{Introduction} \label{introsec}

Calabi--Yau conifolds, i.e., Calabi--Yau 3-folds with only ordinary double points (ODPs), arise naturally in algebraic geometry and string theory, where a Calabi--Yau 3-fold $X$ is a projective Gorenstein $3$-fold with at worst terminal singularities such that $K_X \sim 0$ and $H^1(\mathscr{O}_X) = 0$. For example, every Calabi--Yau 3-fold can be deformed to a smooth one or to a conifold \cite{gros, nast}. M.\ Reid \cite{reid3} had proposed to study the moduli spaces of simply connected smooth Calabi--Yau 3-folds through conifold transitions, by which we mean there is a small projective contraction from a smooth Calabi--Yau to a Calabi--Yau conifold so that the conifold is smoothable. This is usually referred as the \emph{Reid's fantasy}. While non-projective conifold transitions are also considered in the literature, in this paper we stick on the projective ones.

In 1988, P.S. Green and T. H$\ddot{\text{u}}$bsch \cite{gh2} discovered the remarkable connectedness phenomenon: \emph{The moduli spaces of complete intersection Calabi--Yau 3-folds (CICYs) in product of projective spaces are connected with each other by a sequence of conifold transitions}. In \cite[\S3 p.435]{gh2}, the authors deferred the proof of the existence of conifold transitions to a forthcoming paper, which unfortunately has not yet been available. This result had since then been used again and again in the literatures on Calabi--Yau geometry and string theory. While there is no doubt on its significance and correctness, a detailed complete proof to it is still long awaited. The first goal of this paper is to supply such a rigorous proof:

\begin{theorem} [$=$ Theorem \ref{webcicy}] \label{webcicy-intro}
Any two (parameter spaces of) complete intersection Calabi--Yau 3-folds in product of projective spaces are connected by a finite sequence of conifold transitions.
\end{theorem}

In order to connect these parameter spaces of CICY 3-folds, the major idea is to use the determinantal contractions introduced in \cite{cdls}.

Let us recall a standard example to explain the process. Consider the smooth CICY 3-fold $\x$ in $\bP^1 \times \bP^4$ defined by $p^0_j (z) t_0 + p^1_j (z) t_1 = 0$ for $j = 1, 2$, where $t_0, t_1$ are homogeneous coordinates on $\Bbb P^1$, $p^0_1(z), p^1_1(z)$ are two general quartic polynomials and $p^0_2(z), p^1_2(z)$ are two linear polynomials on $\bP^4$. Since $t_i$'s can not both vanish, it must be the case that the determinant
$$
\Delta(z) \coloneqq \det (p^i_j(z))
$$
resulting from the projection along $\bP^1$ vanishes (cf. \S \ref{secde}). If we take $p^i_2(z) = z_i$  for $i =0, 1$ and suitable quartic polynomials $p^0_1(z), p^1_1(z)$, then the quintic $X$ defined by $\Delta(z)$ has 16 ODPs, where $p^i_j(z)$'s vanish simultaneously, along a projective plane in $\bP^4$. Let $\widetilde{X}$ be a smooth quintic in $\bP^4$. Note that all quintic hypersurfaces in $\bP^4$ are deformation equivalent inside a flat family (cf.~Proposition \ref{param}). Thus we get a conifold transition $\x \rightarrow X \rightsquigarrow \widetilde{X}$ which connects parameter spaces of $\x$ and $X$.

In general, the task is to verify that the determinantal contraction is a small resolution of a Calabi--Yau conifold.
The tool used in this paper is the following well known criterion (involved topological constraints) for ODPs.

\begin{proposition} [$=$ Proposition \ref{smalltraodp}] \label{smalltraodp-intro}
Let $\x \rightarrow X$ be a small resolution of a Gorenstein terminal 3-fold $X$ and $\tX$ a smoothing of $X$. Then the difference of the topological Euler numbers $e(\x) - e(\tX)$ equals the number $2 \left| \Sing(X) \right|$ if and only if the singularities of $X$ are ODPs.
\end{proposition}

For 3-dimensional complete intersection varieties in a product of a projective space and a smooth projective variety, we give the formulas of the difference of the Euler numbers and the number of singularities involving Chern classes of vector bundles (cf. Proposition \ref{eulchern0} and Corollary \ref{cicynumsing}).

Another ingredient is a \emph{Bertini-type theorem for vector bundles} (cf.~Theorem \ref{berti}). The necessity for such a result with weaker positivity assumptions comes from the fact that the CICY 3-folds under consideration are not always cut out by \emph{ample divisors}. Combining with the original ideas in \cite{cdls, gh2}, we prove that the singular Calabi--Yau $X$ defined by the determinantal equation (and other equations) has isolated singularities and the determinantal contraction is a small resolution of $X$ (cf.~Theorem \ref{detcon}). According to Proposition \ref{smalltraodp-intro}, it follows that the determinantal contraction is a small resolution of a Calabi--Yau conifold as expected. We also give a formula of the second Betti number of CICYs (cf. Proposition \ref{hodgnum11}).

In the final section, we discuss the relationship between small transitions and conifold transitions. We introduce the primitive small transitions (cf. Definition \ref{pritran}) and prove the following result:

\begin{theorem} [$=$ Theorem \ref{main2}] \label{main2-intro}
Let $\pi : \x \rightarrow X$ be a small projective resolution of a Calabi--Yau 3-fold $X$. If the natural closed immersion $\Def(\widehat{X}) \hookrightarrow \Def(X)$ of Kuranishi spaces is an isomorphism then the singularities of $X$ are ODPs. Moreover, the number of ODPs is equal to the relative Picard number $\rho(\x / X)$.
\end{theorem}

Theorem \ref{main2-intro} is a generalization of the case of relative Picard number one which have been studied in \cite[(5.1)]{gros}. Using the deformation properties of $X$ and $\x$ and the minimal model theory, we will prove it by induction on the relative Picard number.

\medskip

\emph{Acknowledgements.} This article is part of my Ph.D. thesis at National Taiwan University. I am very grateful to my advisor Professor Chin-Lung Wang for his constant encouragement and guidance, and to Taida Institute for Mathematical Sciences and National Taiwan University for providing an excellent environment. I thank Professors Chen-Yu Chi and Hui-Wen Lin for very helpful discussions, Professors Yoshinori Namikawa and Tristan H$\ddot{\mathrm{u}}$bsch for answering my questions, and Professor Bert van Geemen for pointing out a mistake in Example \ref{doubleso}. Finally I thank the Ministry of Science and Technology (MOST, Taiwan) for its financial support.

\section{Preliminaries} \label{prelsec}

Let $\sigma : \sE \rightarrow \sF$ be a morphism of vector bundles of ranks $m$ and $n$ on a variety $M$. Note that there is a natural bijection between morphisms $\sE \rightarrow \sF$ and global sections of $\sE^{\vee} \otimes \sF$.

For $k \leqslant \mathrm{min}(m, n)$, we define the $k$-th degeneracy locus of $\sigma$ by
$$
D_k(\sigma) = \{ x \in M \,|\, \rk(\sigma(x)) \leqslant k \}.
$$
Its ideal is locally generated by $(k + 1)$-minors of a matrix for $\sigma$. We can show that the codimension of $D_k(\sigma)$ in $M$ is less than or equal to $(m - k)(n - k)$ \cite[Theorem 14.4 (b)]{fulto}, which is called its \textit{expected codimension}. Notice that the $0$-th degeneracy locus of $\sigma$ is the zero scheme $Z(\sigma)$.

Now we state a Bertini-type theorem for vector bundles. The following statement is taken from \cite[(2.8)]{ott}.

\begin{theorem} [\cite{ott}] \label{berti}
Let $\sE$ and $\sF$ be vector bundles of ranks $m$ and $n$ on a smooth variety $M$ and let $\sE^{\vee} \otimes \sF$ be generated by global sections. If $\sigma : \sE \rightarrow \sF$ is a general morphism, then one of the following holds:
\begin{enumerate}
  \item\label{berti1} $D_k(\sigma)$ is empty;
  \item\label{berti2} $D_k(\sigma)$ has expected codimension $(m - k)(n -k)$ and the singular locus of $D_k(\sigma)$ is $D_{k - 1}(\sigma)$.
\end{enumerate}
\end{theorem}

Here the "general" means that there is a Zariski open set in the vector space $H^0(\sE^{\vee} \otimes \sF)$ such that each global section $\sigma$ in the open set satisfies (\ref{berti1}) or (\ref{berti2}).

\begin{remark}
Let $D$ be a Cartier divisor on $M$. Assume that the linear system $\Lambda \coloneqq | \sO(D)|$ is base point free. Since the $(- 1)$-th degeneracy locus is empty, the classical Bertini's second theorem follows from Theorem \ref{berti} by taking $k = 0$, $\sE = \sO$ and $\sF = \sO (D)$, i.e., a general member of $\Lambda$ is smooth. We also know, by the Bertini's first theorem, that if $\Lambda$ is not composed of a pencil then its general member is irreducible. However, the general degeneracy locus $D_k(\sigma)$ may not be connected.
\end{remark}

For the case that $\sE$ is a trivial line bundle, $\sigma : \sO_M \rightarrow \sF$ corresponds to a global section of $\sF$. The wedge product by the section gives rise to a complex
$$
\sO_{M} \to \sF \to \wedge^2 \sF \to \cdots \to \wedge^{n - 1}\sF \to \wedge^n \sF.
$$
The dual complex
\begin{equation} \label{kosz}
    K^{\bullet}(\sigma) : \wedge^{n}\sF^{\vee} \to \wedge^{n - 1} \sF^{\vee} \to \cdots \to \wedge^2 \sF^{\vee} \to \sF^{\vee} \xrightarrow{\sigma^{\vee}} \sO_{M}
\end{equation}
is called the \emph{Koszul complex} of $Z(\sigma)$. Note that the image of $\sigma^{\vee}$ is the ideal sheaf of $Z(\sigma)$. We say that $Z(\sigma)$ is \textit{complete intersection} if the sequence
$$
0 \to K^{\bullet}(\sigma) \to \sO_{Z(\sigma)} \to 0.
$$
is exact, that is, the Koszul complex (\ref{kosz}) is a resolution of $\sO_{Z(\sigma)}$. If $M$ is Cohen-Macaulay, then $Z(\sigma)$ is complete intersection if and only if its codimension in $M$ is equal to the rank $n$ of $\sF$ \cite[p.431]{fulto}.

If $M$ is a projective variety, endowed with an ample divisor $\sO(1)$, then the Hilbert polynomial of a complete intersection $Z(\sigma)$ can be computed by the Koszul complex (\ref{kosz}), that is, for $l \in \bZ$
\begin{equation}\label{hilbpol}
    \chi(\sO_{Z(\sigma)} \otimes \sO(l)) = \sum_{i = 0}^n (- 1)^{i} \chi(\wedge^i \sF^{\vee} \otimes \sO(l)).
\end{equation}

A birational morphism is called \emph{small} if the exceptional set has codimension at least two. The following is a simple criterion of ODPs for a small resolution $\x \to X$ if $X$ admits a smoothing.

\begin{proposition} \label{smalltraodp}
Let $\x \rightarrow X$ be a small resolution of a Gorenstein terminal 3-fold $X$ and $\tX$ a smoothing of $X$. Then the difference of the topological Euler numbers $e(\x) - e(\tX)$ equals the number $2 \left| \Sing(X) \right|$ if and only if the singularities of $X$ are ODPs.
\end{proposition}

For the convenience of the reader, we supply a proof here.

\begin{proof}
Let $C_i$ be the exceptional curve over an isolated hypersurface singularity $p_i$. We have the identity of the topological Euler numbers (for a proof see \cite[Theorem 7]{ross11})
\begin{equation*}
e(\x) - e(\tX) = \sum m(p_i) + \sum \left(e(C_i) - 1 \right),
\end{equation*}
where $m(p_i)$ is the Milnor number of $p_i$. According to \cite[Proposition 1]{pinkh}, the exceptional curve $C_i$ is a union of smooth rational curves which meet transversally and thus the number $e(C_i) - 1$ is equal to $n_i$ the number  of irreducible components of $C_i$. Observe that $m(p_i)$ and $n_i$ are greater than or equal to one.  Then $\sum m(p_i) + \sum n_i \geqslant 2 \left| \Sing(X) \right|$, and the equality holds if and only if $n_i = m(p_i) = 1$ for all $i$.
\end{proof}

The following lemma will be used in Remark \ref{genermk} and the proof of Theorem \ref{detcon}.

\begin{lemma} \label{prozari}
The product Zariski topology on $\bC^n \times \bC^m$ is strictly coarser than the Zariski topology on $\bC^{n + m}$.
\end{lemma}

\begin{proof}[Sketch of proof]
Let $I$ and $J$ be ideals of $\bC [x_1, \cdots, x_n]$ and $\bC [y_1, \cdots, y_m]$, and let $I^e$ and $J^e$ be ideals generated by $I$ and $J$ in $\bC [x_1, \cdots, x_n, y_1, \cdots, y_m]$ respectively. Let $V(I)$ and $V(J)$ denote Zariski closed subsets defined by $I$ and $J$ respectively. Then the standard open subset $(\bC^n \setminus V(I)) \times (\bC^m \setminus V(J))$ of the product topology is the complement set of the Zariski closed subset $V(I^e) \cap V(J^e)$ of $\bC^{n + m}$. Remark that not every open subset in the Zariski topology on $\bC^{n + m}$ is open in the product Zariski topology.
\end{proof}

We conclude this section with an elementary lemma, which will be used in the proof of Theorem \ref{main2}.

\begin{lemma} \label{diff}
Let $C = \bigcup C_i$ be a curve in a smooth $3$-fold $Y$ such that the irreducible components $C_i$ meet in a finite set of points. Then there exists an injection $\bigoplus_i H^2_{C_i}(Y, \Omega^2_Y) \hookrightarrow H^2_C(Y, \Omega^2_Y)$. Moreover, it is an isomorphism if $C_i$ are mutually disjoint.
\end{lemma}

\begin{proof}
By induction on the number of components of $C$, we may assume that $C = C_1 \cup C_2$. From the Mayer--Vietoris sequence, we get
$$
H^2_{C_1 \cap C_2}(\Omega^2_Y) \rightarrow H^2_{C_1}(\Omega^2_Y) \oplus H^2_{C_2}(\Omega^2_Y) \rightarrow H^2_{C}(\Omega^2_Y).
$$
Since $\Omega^2_Y$ is locally free and $\mathrm{depth}_{C_1 \cap C_2} \mathscr{O}_Y = 3$, the local cohomology group $H^2_{C_1 \cap C_2}(\Omega^2_Y)$ vanishes (cf. \cite[III Ex.3.4]{hart}), which completes the proof.
\end{proof}

\section{Configurations and Parameter Spaces}\label{cp}

We start by introducing the configuration of \emph{complete intersection} varieties of dimension $d$ and constructing their parameter spaces.

A configuration of dimension $d$ is a pair $[ V \| \mathfrak{L}]$ of a smooth projective variety $V$ with $\dim V = m + d$ and a sequence of line bundles $\mathfrak{L} = (\mathscr{L}_1, \cdots, \mathscr{L}_m)$, where $\sL_j$ is generated by global sections. Let $X$ be a $d$-dimensional variety. The variety $X$ is said to be a member of the configuration, denoted by $X \in [ V \| \mathfrak{L}]$, if it is defined by global sections $\sigma_j$ of $\sL_j$ for $1 \leqslant j \leqslant m$.

If
\begin{center}
    $V = \prod_{i = 1}^k \bP^{n_i}$ and $\sL_j = \bigotimes_{i = 1}^k pr_i^{\ast} \sO_{\bP^{n_i}} (q^i_j)$,
\end{center}
where $pr_i : V \rightarrow \bP^{n_i}$ is the natural projection and $q^i_j \geqslant 0$ for all $i, j$, then we rewrite $[ V \| \mathfrak{L}]$ as a configuration matrix
\begin{equation}\label{conmat}
[\mathbf{n} \| \mathbf{q}]
=
\left[
\begin{array}{c||ccc}
n_1 & q^1_1 & \cdots & q^1_m \\
\vdots & \vdots & \ddots & \vdots \\
n_k & q^k_1 & \cdots & q^k_m
\end{array}
\right].
\end{equation}
The $(q^1_j, \cdots, q^k_j)$ and $\mathbf{q}$ will be called the multidegree of the line bundle $\sL_j$ and a member $X \in [\mathbf{n} \| \mathbf{q}]$ respectively. We may assume that
$$
\sum\nolimits_{i = 1}^k q^i_j \geqslant 2
$$
for all $1 \leqslant j \leqslant m$ (otherwise a hyperplane section of only one factor $\bP^n$ reduces the factor to $\bP^{n - 1}$). Note that the global sections of $\sL_j$ are multi-homogeneous polynomials of multidegree $(q^1_j, \cdots, q^k_j)$.

Two configuration matrices are said to represent the same configuration if one can go from one to the other by a permutation of the rows or of the columns other than first. We say that $[\mathbf{n}_1 \| \mathbf{q}_1]$ is a sub-configuration matrix of $[\mathbf{n} \| \mathbf{q}]$ if
$$
\left[
\begin{array}{c||cc}
\mathbf{n}_1 & \mathbf{q}_1 & \mathbf{a} \\
\mathbf{m} & \mathbf{0} & \mathbf{b}
\end{array}
\right]
$$
and $[\mathbf{n} \| \mathbf{q}]$ represent the same configuration.

In the case $V = \prod_{i = 1}^k \bP^{n_i}$, we can explain the meaning of a complete intersection $X \in [\mathbf{n} \| \mathbf{q}]$ precisely by defining a projective family for the configuration $[\mathbf{n} \| \mathbf{q}]$ whose fibers are complete intersections of multidegree $\mathbf{q}$.

In the following we will write $\underline{T}_i$ and $\underline{u}$ as a short form for indeterminates $T_{i0}, \cdots, T_{i n_i}$ and $u_1, \cdots, u_a$ respectively. Set $R = \bC [\underline{T}_1; \cdots; \underline{T}_k]$.

Let $X$ be a complete  intersection variety defined by a sequence of multi-homogeneous polynomials $\sigma= (\sigma_j)$ of multidegree $\mathbf{q}$ and of dimension $d$.
Let $$\Phi^{(1)}, \cdots, \Phi^{(a)}$$ be a basis of $\bigoplus_{j = 1}^m H^0(V, \sL_j)$ and write $\Phi^{(h)} = (\phi^{(h)}_j)$ where $\phi^{(h)}_j \in R$ with multidegree $(q^1_j, \cdots, q^k_j)$.

Let $K_{\bullet} \coloneqq K_{\bullet}(\sigma + \sum_{h = 1}^a u_h \Phi^{(h)})$ be the Koszul complex (cf. (\ref{kosz})) and
$$
D \coloneqq \mathrm{Supp}(H_1(K_{\bullet})) \subseteq \mathbb{A}^{a + N} = \Spec (\bC [\underline{u}; \underline{T}_1; \cdots; \underline{T}_k])
$$
where $N = \sum_i n_i + k$.

Set $q$ be the projection from $\mathbb{A}^{a + N}$ onto $\mathbb{A}^a$. Then $U \coloneqq q(\mathbb{A}^{a + N} \setminus D)$ is the open set of points $\underline{u} \in \mathbb{A}^a$ with $K_{\bullet}(\sigma + \sum_{h = 1}^a u_h \Phi^{(h)})$ being exact. According to that $X = Z(\sigma)$ is a complete intersection, it follows that $U$ contains the origin. Let
$$
I = \left< \left. \sigma_l + \sum_h u_h \phi^{(h)}_l \, \right| \, 1 \leqslant l \leqslant m \right>
$$
and
$
\mathscr{X} = \mathrm{Proj}\left(R [\underline{u}] / I\right).
$

Consider the projection $P : \mathscr{X} \subseteq V \times \mathbb{A}^a \rightarrow \mathbb{A}^a$ and its restriction $P_U : \mathscr{X}_U \rightarrow U$. Since the the Koszul complex $K_{\bullet}$ is exact on $U$, all fibers of $P_U$ are complete intersections of multidegree $\mathbf{q}$ and have the same Hilbert polynomial $P(t)$ which is computed by the Koszul resolution and depends on its multidegree (cf. (\ref{hilbpol})). Hence $P_U$ is a flat family with the fiber $\mathscr{X}_0 = X$ \cite[III Theorem 9.9]{hart}.

To summarize what we have proved, we get the following proposition:

\begin{proposition} \label{param}
Let $X$ be a variety. If $X \in [\mathbf{n} \| \mathbf{q}]$ then there is a Zariski open set $U$ in $H^0(V, \bigoplus_{j = 1}^m \sL_j)$, a closed point $t_0 \in U$ and a flat projective morphism $P_U : \mathscr{X}_U \rightarrow U$ with the fiber $\mathscr{X}_{t_0} = X$ such that all complete intersections in $V$ of multidegree $\mathbf{q}$ are parameterized by the pair $(U, P_U)$.
\end{proposition}

Hence we may use the configuration $[\mathbf{n} \| \mathbf{q}]$ to denote the parameter space of $d$-dimensional  complete intersections in $V$ of multidegree $\mathbf{q}$.

To point out what the fundamental cycle and the normal bundle of a smooth complete intersection is, we state the following result by using Theorem \ref{berti}.

\begin{proposition} \label{gene1}
Let $\mathfrak{L} = (\mathscr{L}_1, \cdots, \mathscr{L}_m)$ be a sequence of globally generated line bundles over a smooth projective variety $V$ and $[V \| \mathfrak{L}]$ a configuration of dimension $d$. Then there is a Zariski open subset $U$ in $H^0(V, \bigoplus_{j = 1}^m \sL_j)$ such that $X = Z(\sigma)$ is smooth and of dimension $d$ for every element $\sigma$ in $U$. Moreover, the normal bundle of $X$ in $V$ is $\bigoplus_{j = 1}^m \sL_j |_X$ and the fundamental class $[X]$ in $A_d(V)$ is the top Chern class of $\bigoplus_{j = 1}^m \sL_j$.
\end{proposition}

\begin{proof}
Applying Theorem \ref{berti} to the case $k = 0$, $\sE = \sO_V$ and $\sF = \bigoplus_{j = 1}^m \sL_j$, the zero locus $Z(\sigma)$ is smooth and has the expected codimension $m$ for a general $\sigma : \sE \rightarrow \sF$. Namely, there is a Zariski open subset $U$ in $H^0(V, \sF)$ such that every element $\sigma$ in $U$ defines a smooth complete intersection $X$ in $V$ of dimension $\dim V - m = d$. By \cite[Example 3.2.16]{fulto}, the fundamental class of a general member in $A_d(V)$ is $c_{m}(\sF) \cap [V]$.
\end{proof}

\begin{remark} \label{genermk}
Since $\bigoplus_{j = 1}^m H^0(V, \sL_j)$ is naturally isomorphic to $H^0(V, \sF)$ (as vector spaces), any element $\sigma$ in $U$ corresponds $(s_j)$ in $\bigoplus_{j = 1}^m H^0(V, \sL_j)$ and thus $X = Z(\sigma)$ is the complete intersection $\cap_{j = 1}^m Z(s_j)$.

Applying Theorem \ref{berti} repeatedly, we may assume that, for a general section $(s_j)$ in $H^0(V, \sF)$, the divisor $Z(s_j)$ is smooth with all subsets of the $Z(s_j)$'s meeting transversely. For example, by Theorem \ref{berti}, there is a Zariski open subset $V_j$ of $H^0(V, \sL_j)$ such that $Z(s_j)$ is smooth for $1 \leqslant j \leqslant m$. According to Lemma \ref{prozari}, it follows that $\prod_{j =1}^m V_j$ is Zariski open in $H^0(V, \sF)$. Replacing $U$ by $U \cap (\prod_{j =1}^m V_j)$, which is Zariski open in $H^0(V, \sF)$, all $Z(s_j)$'s are smooth for $(s_j) \in U$.
\end{remark}

In order to connect two configurations, we shall define a formal correspondence on configurations which is introduced in \cite{cdls}.

Let $P$ be a smooth projective variety, and let $p$ and $\pi$ be the projections from $\bP^n \times P$ onto $\bP^n$ and $P$ respectively. If $V = \bP^n \times P$ and
\begin{equation*}
\hsL_j =
\left\{
\begin{array}{ll}
p^{\ast}\sO_{\bP^n}(1) \otimes \pi^{\ast}\sL_j & \mbox{if } 1 \leqslant j \leqslant n +1, \\
\pi^{\ast}\sL_j  & \mbox{if } n + 2 \leqslant j \leqslant m,
\end{array}
\right.
\end{equation*}
where $\sL_j$'s are globally generated line bundles on $P$, we rewrite the configuration $\hsC \coloneqq [V \| \hsL_1, \cdots, \hsL_m]$ as
$$
\left[
\begin{array}{c||cccccc}
n & 1 & \cdots & 1 & 0 & \cdots & 0 \\
P & \sL_1 & \cdots & \sL_{n + 1} & \sL_{n + 2} & \cdots & \sL_m
\end{array}
\right].
$$
We introduce a new configuration
$$
\sC \coloneqq
\left[
\begin{array}{c||ccccc}
P & \bigotimes_{i = 1}^{n + 1} \sL_i & \sL_{n + 2} & \cdots & \sL_m
\end{array}
\right]
$$
so as to remove the $\bP^n$ factor and denote the \emph{formal correspondence} by
\begin{equation}\label{fsc}
\hsC \mathrel{\leftarrow\mkern-14mu\leftarrow}\!\rightarrow \sC.
\end{equation}
Remark that the paper \cite{cdls} refers to the correspondence of passing from the right hand side to the left as \emph{splitting} and the reverse process as \emph{contraction}.

We are going to compute the difference of topological Euler numbers of smooth members under the formal correspondence of $3$-dimensional configurations.

\begin{proposition} \label{eulchern0}
Let $S$ be a smooth variety of dimension $4$,
$$
\mathscr{R} =
\left[
\begin{array}{c||ccc}
n & 1 & \cdots & 1  \\
S & \sL_1 & \cdots & \sL_{n + 1}
\end{array}
\right]
$$
a configuration of dimension 3 and $\sE = \bigoplus_{i = 1}^{n + 1} \sL_i$ a vector bundle of rank $n + 1$. Assume that $\x \in \mathscr{R}$ and $\tX \in
\left[
\begin{array}{c||c}
S & \bigotimes_{i = 1}^{n + 1} \sL_i
\end{array}
\right]$ are smooth members. Then we have
$$
e(\x) - e(\tX) = 2 \int_S \left(c_2(\sE)^2 - c_1(\sE) c_3(\sE)\right)
$$
where $e(-)$ denotes the topological Euler number.
\end{proposition}

\begin{proof}
Let $\iota : \tX \hookrightarrow S$ be the inclusion. Since $S$ and $\tX$ are smooth, we have the normal exact sequence
$$
0 \rightarrow T_{\tX} \rightarrow T_S |_{\tX} \rightarrow N_{\tX} \rightarrow 0.
$$
By Proposition \ref{gene1}, the normal bundle of the hypersurface $\tX$ in $S$ is $\otimes_{i = 1}^{n + 1} \sL_i$. Let
\begin{equation}\label{poft}
p(t) \coloneqq \iota_{\ast} c_t(T_{\tX}) =c_t(T_S) s_t(\otimes_{i = 1}^{n + 1} \sL_i),
\end{equation}
where $c_t(\mathscr{V})$ is the Chern polynomial of a vector bundle $\mathscr{V}$ and $c_t(\mathscr{V})s_t(\mathscr{V}) = 1$. Observe that $\otimes_{j = 1}^{n + 1} \sL_j$ and $\sE$ have the same first Chern class $\sum_{j = 1}^{n + 1} c_1(\sL_j)$. Then the fundamental class of $\tX$ in $A_3(S)$ is
\begin{equation}\label{fundcyc}
c_1(\otimes_{j = 1}^{n + 1} \sL_j) \cap [S] = c_1(\sE) \cap [S].
\end{equation}

We are going to calculate  $e(\tX)$. According to that $s_t(\otimes_{j = 1}^{n + 1} \sL_j)$ is the inverse of the Chern polynomial $c_t(\otimes_{j = 1}^{n + 1} \sL_j) = 1 + c_1(\otimes_{j = 1}^{n + 1} \sL_j)$, it follows that
\begin{equation} \label{segpol}
s_t(\otimes_{j = 1}^{n + 1} \sL_j) = \sum_{i = 0}^{\infty} (-c_1(\sE))^i t^i =\sum_{i = 0}^{\infty} s_1(\sE)^i t^i.
\end{equation}
Set $c_t(S) = c_t(T_S)$. Using (\ref{segpol}) and collecting the coefficient of $t^3$ in (\ref{poft}), we get
\begin{equation} \label{p3}
\frac{1}{3 !}p'''(0) = s_1(\sE)^3 + c_1(S) s_1(\sE)^2 + C_s,
\end{equation}
where
\begin{equation} \label{ccher}
C_s \coloneqq c_2(S) s_1(\sE) + c_3(S).
\end{equation}
By (\ref{fundcyc}), (\ref{p3}), and the Gauss-Bonnet theorem, we get
$$
e(\tX) = \int_{\tX} c_3(\tX) = \int_S \frac{1}{3 !} p'''(0) c_1(\sE).
$$

Let $pr$ be the projection from $\bP^n \times S$ onto $\bP^n$ and $\hiota$ the inclusion from $\x$ into $\bP^n \times S$. To compute $e(\x)$, we identify the line bundle $\sO_{\bP^n}(1)$ on $\bP^n$ with $pr^{\ast} \sO_{\bP^n}(1)$ on $\bP^n \times S$ (similarly for vector bundles on $S$).

According to $N_{\x} = \left(\sE \otimes \sO_{\bP^n}(1)\right) |_{\tX}$, it follows that
$$
q(t) \coloneqq \hiota_{\ast} c_t(T_{\x}) = c_t(T_{\bP^n} \oplus T_S) s_t(\sE \otimes \sO_{\bP^n}(1)).
$$
By Proposition \ref{gene1}, the fundamental class of $\x$ in $A_3(\bP^n \times S)$ is
$$
c_{n + 1}(\sE \otimes \sO_{\bP^n}(1)) \cap [\bP^n \times S].
$$
Set $H = c_1(\sO_{\bP^n}(1))$ in $A_1(\bP^n \times S)$. According to $c_t(T_{\bP^n} \oplus T_S) = c_t(T_{\bP^n}) c_t(T_S)$, it follows that the coefficient of $t^3$ in $q(t)$ is
\begin{equation} \label{coefqt}
\frac{1}{3 !} q'''(0) = \sum_{p = 0}^3 \binom{n + 1}{p} \left[ \sum_{i + j = 3 - p} c_i(S) s_j(\sE \otimes \sO(1)) \right] H^p.
\end{equation}
From \cite[Example 3.1.1]{fulto}, we have
\begin{equation} \label{segtensor}
s_l(\sE \otimes \sO(1)) = \sum_{i = 0}^l (- 1)^{l - i} \binom{n + l}{n + i} s_i(\sE) H^{l - i}.
\end{equation}
By substituting (\ref{segtensor}) into (\ref{coefqt}), we obtain
\begin{equation} \label{coefqt2}
\frac{1}{3 !} q'''(0) = s_1(\sE) H^2 - \left[ 2 s_2(\sE) + c_1(S) s_1(\sE)\right] H + \left[s_3(\sE) + c_1(S) s_2(\sE) + C_s \right],
\end{equation}
where $C_s$ is the class as defined in (\ref{ccher}). (For example, the coefficient of $H^3$ in (\ref{coefqt2}) is
$
\binom{n + 1}{3} - (n + 1) \binom{n + 1}{2} + (n + 1) \binom{n + 2}{n} - \binom{n + 3}{n}
$
which is equal to zero.)

We regard the class $\frac{1}{3 !} q'''(0) c_{n + 1}(\sE \otimes \sO_{\bP^n}(1))$ as a polynomial in $H$, denoted it by $Q(H)$. Then
$$
e(\x) = \int_{\bP^n \times S} Q(H) = \int_S \frac{1}{n !} Q^{(n)}(0).
$$
If we can prove
\begin{equation} \label{qmp}
\frac{1}{n !} Q^{(n)}(0) - \frac{1}{3 !} p'''(0) c_1(\sE) = 2 \left(c_2(\sE)^2 - c_1(\sE) c_3(\sE)\right),
\end{equation}
then the Proposition follows by integrating the equality (\ref{qmp}) over $S$.

Note that the top Chern class of $\sE \otimes \sO_{\bP^n}(1)$ is $\sum_{i = 0}^{n + 1} c_{n + 1 - i}(\sE) H^i$. According to (\ref{coefqt2}), it follows that the coefficient $\frac{1}{n !} Q^{(n)}(0)$ of $H^n$ in $Q(H)$ is equal to the coefficient of $H^n$ in the class
\begin{equation*}
\frac{1}{3 !} q'''(0) \left(c_1(\sE) H^n + c_2(\sE) H^{n - 1} + c_3(\sE) H^{n - 2}\right).
\end{equation*}
From  $H^{n + 1} = 0$ in $A_{\ast}(\bP^n)$, it follows that
\begin{multline*}
\frac{1}{n !} Q^{(n)}(0) = s_1(\sE) c_3(\sE) - [ 2 s_2(\sE)  + c_1(S) s_1(\sE) ] c_2(\sE) \\
+ \left[s_3(\sE) + c_1(S) s_2(\sE) + C_s \right] c_1(\sE).
\end{multline*}
Rewriting Segre classes $s_1(\sE)^i$ in $1/ 3 ! p'''(0) c_1(\sE)$ and $s_i(\sE)$ in terms of Chern classes $c_i(\sE)$ by using the recurrence relations $c_l(\sE) = - \sum_{i =1}^{l} s_i(\sE) c_{n - i}(\sE)$, we obtain that the left-hand side of (\ref{qmp}) equals to
$$
2 [ s_3 (\sE) c_1(\sE) - s_2(\sE) c_2(\sE) + s_2(\sE) c_1(\sE)^2],
$$
which is nothing but the right-hand side of (\ref{qmp}).
\end{proof}

\begin{corollary} \label{eulchern}
Let $\hsC \mathrel{\leftarrow\mkern-14mu\leftarrow}\!\rightarrow \sC$ be as in (\ref{fsc}), and let $\sE = \bigoplus_{i = 1}^{n + 1} \sL_i$ and $\sF = \bigoplus_{i = n + 2}^{m} \sL_i$ be vector bundles of rank $n + 1$ and $m - n - 1$ respectively. Assume that $\hsC$ and $\sC$ are configurations of dimension 3. Given smooth members $\x \in \hsC$ and $\tX \in \sC$, we assume that there is a global section $\sigma$ of $\sF$ such that the zero locus $Z(\sigma)$ is smooth of dimension $4$ and contains $\tX$. Then
$$
e(\x) - e(\tX) = 2 \int_P \left(c_2(\sE)^2 - c_1(\sE) c_3(\sE)\right) c_{m - n -1}(\sF).
$$
\end{corollary}

\begin{proof}
Let $S$ be the zero locus $Z(\sigma)$. The corollary follows immediately from Proposition \ref{eulchern0} and the fundamental class of $S$ in $A_4(P)$ is $c_{m - n - 1}(\sF) \cap [P]$.
\end{proof}

\begin{example}
Consider
$$
\widehat{\mathscr{C}} \coloneqq
\left[
\begin{array}{c||ccc}
2 & 1 & 1 & 1 \\
3 & 1 & 1 & 2 \\
1 & 0 & 0 & 2
\end{array}
\right]
\mathrel{\leftarrow\mkern-14mu\leftarrow}\!\rightarrow
\mathscr{C} \coloneqq
\left[
\begin{array}{c||c}
3 & 4 \\
1 & 2
\end{array}
\right].
$$
For smooth member $\x \in \widehat{\mathscr{C}}$ and $\tX \in \mathscr{C}$, the Euler numbers $e(\x)$ and $e(\tX)$ are $-112$ and $-168$ respectively. Let $s$ (resp. $t$) be the class of a hyperplane on $\bP^3$ (resp. $\bP^1$), and let $\mathscr{E}$ be the vector bundle $\sO(1,0) \bigoplus \sO(1,0) \bigoplus \sO(2,2)$ of rank 3 on $\bP^3 \times \bP^1$. Then the Chern classes of $\mathscr{E}$ are
\begin{center}
    $c_1(\mathscr{E}) = 4 s + 2 t$, $c_2(\mathscr{E}) = 5 s^2 +4 s t$, $c_3(\mathscr{E}) = 2 s^3 + 2 s^2 t$,
\end{center}
and the coefficient of $s^3 t$ in $c_2(\mathscr{E})^2 - c_1(\mathscr{E}) c_3(\mathscr{E})$ is $28$.
\end{example}

\section{Calabi--Yau Configurations} \label{cyc}

From now on we will suppose that all configurations are of dimension 3.

\begin{definition}
A configuration matrix $[\mathbf{n} \| \mathbf{q}]$ is called a complete intersection Calabi--Yau (CICY) configuration if it satisfy the Calabi--Yau condition
$$
\sum\nolimits_{j = 1}^m q^i_j = n_i + 1
$$
for all $1 \leqslant i \leqslant k$.
\end{definition}

It is easy to see that CICY configuration matrices are preserved under formal correspondences (\ref{fsc}). Note that the topological Euler number of a smooth member which belongs to a CICY configuration matrix is non-positive \cite[(2.28)]{cdls}.

\begin{remark}\label{ndiag}
We do not allow that a Calabi--Yau 3-fold $X$ is a product of three elliptic curves or of an elliptic curve and $K3$ surface since $H^1(\sO_X) = 0$. Further we are not interested in a configuration matrix which contains the sub-configuration $[1 \| 2]$ because the sub-configuration describes two points (counted with multiplicity) in $\bP^1$. To exclude such cases, we only treat \emph{non block-diagonal} CICY configuration matrices.
\end{remark}

Let us consider the simple case for all $n_i = 1$ and $q^i_j = 0$ or $2$.

\begin{example}\label{simple}
Given a CICY configuration $k \times (m + 1)$-matrix $[\mathbf{n} \| \mathbf{q}]$ with $n_i = 1$ and $q^i_j = 0$ or $2$ for all $i, j$, we have $k = m+ 3$. By Remark \ref{ndiag}, we know that $[\mathbf{n} \| \mathbf{q}]$ is non block-diagonal and thus
$$
\sum\nolimits_{i = 1}^k q^i_j \geqslant 4
$$
for each column of $\mathbf{q}$. According to the Calabi--Yau condition, it follows that
$$
4(k - 3) \leqslant \sum\nolimits_{i,j} q^i_j = 2 k
$$
and therefore $4 \leqslant k \leqslant 6$. When $k$ equals $5$ or $6$, we get a product of an elliptic curve and $K3$ surface or of three elliptic curves respectively. By Remark \ref{ndiag}, the CICY configuration matrix must be
$$
\left[
\begin{array}{c||c}
1 & 2 \\
1 & 2 \\
1 & 2 \\
1 & 2
\end{array}
\right]
$$
in this simple case. We denote this configuration matrix by $\mathscr{C}_{1111}$.
\end{example}

We say that a configuration connects to another formally if, after finite formal correspondences (\ref{fsc}), one represents the same configuration as the another one. The following proposition were proved in \cite[Lemma 2]{gh2}, for the convenience of the readers we recall the proof here.

\begin{proposition} [\cite{gh2}] \label{fconn}
Every CICY configuration matrices can be connected formally.
\end{proposition}

\begin{proof}
Given a (non block-diagonal) CICY configuration matrix $[\mathbf{n} \| \mathbf{q}]$ as in (\ref{conmat}). We perform formal correspondences iteratively until we arrive at a configuration matrix for which each row entries $q^i_j$ with $n_i > 1$ are $0$ or $1$ (for example, introducing a sub-configuration matrix $[1 \| 1 1]$ to split it). Perform next formal correspondences in a way that finally leaves each $n_i = 1$ and $q^i_j = 0$ or $2$. Notice that non block-diagonal CICY configuration matrices are preserved under formal correspondences. According to Example \ref{simple}, it follows that the configuration matrix is the simple configuration $\mathscr{C}_{1111}$.
\end{proof}

\begin{remark} \label{fconnrmk}
To illustrate Proposition \ref{fconn}, we give formal correspondences connecting the configuration of quintic hypersurfaces in $\bP^4$ to $\mathscr{C}_{1111}$:
\begin{align*}{\small
\left[
\begin{array}{c||c}
4 & 5
\end{array}
\right]
\mathrel{\leftarrow\mkern-14mu\leftarrow}\!\rightarrow
\left[
\begin{array}{c||cc}
4 & 4 & 1 \\
1 & 1 & 1 \\
\end{array}
\right]
\mathrel{\leftarrow\mkern-14mu\leftarrow}\!\rightarrow
\left[
\begin{array}{c||ccc}
4 & 3 & 1 & 1 \\
1 & 1 & 1 & 0 \\
1 & 1 & 0 & 1 \\
\end{array}
\right]
\mathrel{\leftarrow\mkern-14mu\leftarrow}\!\rightarrow
\left[
\begin{array}{c||cccc}
4 & 2 & 1 & 1 & 1 \\
1 & 1 & 1 & 0 & 0 \\
1 & 1 & 0 & 1 & 0 \\
1 & 1 & 0 & 0 & 1 \\
\end{array}
\right]
}
\end{align*}
\begin{align*} {\small
\mathrel{\leftarrow\mkern-14mu\leftarrow}\!\rightarrow
\left[
\begin{array}{c||ccccc}
4 & 1 & 1 & 1 & 1 & 1 \\
1 & 1 & 1 & 0 & 0 & 0 \\
1 & 1 & 0 & 1 & 0 & 0 \\
1 & 1 & 0 & 0 & 1 & 0 \\
1 & 1 & 0 & 0 & 0 & 1\\
\end{array}
\right]
\mathrel{\leftarrow\mkern-14mu\leftarrow}\!\rightarrow
\left[
\begin{array}{c||c}
1 & 2 \\
1 & 2 \\
1 & 2 \\
1 & 2
\end{array}
\right].
}
\end{align*}
\end{remark}

The following proposition is an application of Theorem \ref{berti}, which is a well known result in \cite{gh1}. The remaining task is to prove that a general CICY member is irreducible and $H^1(\sO) = 0$ by using a suitable Lefschetz-type theorem for an ample reducible divisor.

\begin{proposition} \label{gene2}
A general member of a CICY configuration matrix is a smooth Calabi--Yau 3-fold.
\end{proposition}

\begin{proof}
Let $V = \prod_{i = 1}^k \bP^{n_i}$ and let $[\mathbf{n} \| \mathbf{q}]$ be a (non block-diagonal) CICY configuration matrix. Let $\sL_j$ be the line bundle with the multidegree $(q^1_j, \cdots, q^k_j)$. Note that the canonical bundle of $X$ is trivial by the adjunction formula. By Proposition \ref{gene1}, it suffices to prove that a general smooth member $X \in [\mathbf{n} \| \mathbf{q}]$ is connected and $H^1(\sO_X) = 0$.  Namely, we only need to prove that $H^0(X, \bC)$ and $H^1(X, \bC)$ have dimension one and zero respectively.

Pick a general section $(s_j)$ in $H^0(V, \bigoplus_{j = 1}^m \sL_j)$ for which the divisor $D_j \coloneqq Z(s_j)$ is a smooth with all subsets of the $D_j$'s meeting transversely (cf. Remark \ref{genermk}). We notice that if all $q^{i_s}_{j} = 0$ for some $i_s$ then $D_J$ is of the form $D_{J}' \times \bP^{n_{i_s}}$ where $D_{J}'$ is a complete intersection in $\prod_{i \neq i_s} \bP^{n_i}$.  In particular, $H^0(D_j, \bC) $ and $H^1(D_j, \bC) $ are one and zero respectively, by Lefschetz hyperplane theorem, and thus $D_j$ is irreducible for all $1 \leqslant j \leqslant m$.

Using the mixed Hodge theory and Lefschetz hyperplane theorem on the ample divisor $\sum_{j = 1}^m D_j$, we get exact sequences \cite[(2.1)]{chs}, for $i =0, 1$,
\begin{equation}\label{lef}
0 \rightarrow H^i(V, \bC) \cdots \rightarrow \bigoplus\nolimits_{|J| = r} H^i(D_J, \bC) \rightarrow \cdots  H^i(X, \bC) \rightarrow 0
\end{equation}
where $D_J \coloneqq D_{j_1} \bigcap \cdots \bigcap D_{j_r}$ for a multi-index $J = (j_1, \cdots, j_r)$  of length $|J| = r$ with $1 \leqslant j_1 < \cdots < j_r \leqslant m$ and $X = \bigcap_{|J| = m } D_J$. Note that $i + m < \dim V$ for $i = 0, 1$.

By induction, it follows that the dimension of $\bigoplus_{|J| = r} H^0(D_J, \bC)$ is $\binom{m}{r}$ and of $\bigoplus_{|J| = r} H^1(D_J, \bC)$ is zero for the length $r < m$. We remark that the induction process works because every $D_J$ has the form $D_{J}' \times \prod \bP^{n_l}$ with $D' = \sum D_j'$ is ample. Hence the connectedness and simple connectedness of $D_J$ can be proved in the similar way as shown before. Using the sequence (\ref{lef}) and dimension counting, we get the dimension of $H^0(X, \bC)$ and $H^1(X, \bC)$ are one and zero respectively.
\end{proof}

As a byproduct of the proof of Proposition \ref{gene2}, we obtain the following second Betti number formula:

\begin{proposition} \label{hodgnum11}
With the notation as in the proof of Proposition \ref{gene2},
$$
b_2(X, \bC) = (-1)^m \left(m + \sum_{r = 1}^{m - 1} (- 1)^{r} \sum_{|J| = r} b_2(D_J, \bC)\right).
$$
Moreover, the second Betti number of $X$ equals the second Betti number of the ambient space $V$ if $b_2(D_J, \bC) = b_2(V, \bC)$ for each $1 \leqslant |J| < m$.
\end{proposition}

\begin{proof}
By $V = \prod_{i = 1}^m \bP^{n_i}$ and K\"{u}nneth formula, the second Betti number of $V$ equals $m$. Since $\dim V > m + 2$, the exact sequence (\ref{lef}) holds for $i = 2$ and the proposition follows.
\end{proof}

\begin{remark}
For a smooth Calabi--Yau 3-fold $X$, the topological Euler number $e(X)$ is $2 (h^{1, 1}(X) - h^{2, 1}(X))$. To know the Hodge number $h^{1, 1}(X)$ and $h^{2, 1}(X)$, it suffices to compute either one of these two Hodge numbers and $e(X)$. In \cite{ghla}, finding these Hodge numbers corresponding to a given CICY configuration matrix is in principle just a matter of looking up the \emph{relevant} matrix in the list. Those calculated in \cite{ghla} for the 7868 CICY matrices constructed in \cite{cdls}. Proposition \ref{hodgnum11} gives a direct calculation of $h^{1, 1}(X) = b_2(X, \bC)$ for $X$ in \emph{any} given CICY configuration matrix.
\end{remark}

The remark is illustrated by the following example which was given in the appendix of \cite{ghla}.

\begin{example} [\cite{ghla}]
Consider
$$
X \in
\left[
\begin{array}{c||ccccc}
4 & 3 & 1 & 1 & 0 & 0 \\
2 & 0 & 1 & 0 & 1 & 1 \\
2 & 0 & 0 & 1 & 1 & 1
\end{array}
\right].
$$
Applying Lefschetz hyperplane theorem, K\"{u}nneth formula and Proposition \ref{hodgnum11}, we get
$$
b_2(X, \bC) = b_2(\bP^4, \bC) + b_2(D, \bC)
$$
where
$D \in
\left[
\begin{array}{c||cc}
2 &  1 & 1 \\
2 &  1 & 1
\end{array}
\right]
$
is a smooth surface with Euler number 6. Therefore $b_2(D) = e(D) - 2 =4$ and the second Betti number of $X$ is 5.
\end{example}

\section{Connecting the CICY Web via Determinantal Contractions}\label{secde}

We first recall the definition of determinantal contractions, which is introduced in \cite{cdls}, between configurations of complete intersection varieties in a product of a projective space and a smooth projective variety.

Let $P$ be a smooth projective variety, and let $\hsC$ be a configuration of dimension $3$ of the type
$$
\hsC =
\left[
\begin{array}{c||cccccc}
n & 1 & \cdots & 1 & 0 & \cdots & 0 \\
P & \sL_1 & \cdots & \sL_{n + 1} & \sL_{n + 2} & \cdots & \sL_m
\end{array}
\right]
$$
where $\sL_j$'s are line bundles on $P$. Note that $\dim P = m- n + 3$ because $\hsC$ is of dimension $3$. We have the formal correspondence
$$
\hsC
\mathrel{\leftarrow\mkern-14mu\leftarrow}\!\rightarrow
\sC =
\left[
\begin{array}{c||ccccc}
P & \bigotimes_{i = 1}^{n + 1} \sL_i & \sL_{n + 2} & \cdots & \sL_m
\end{array}
\right].
$$

Let $\pi : \bP^n \times P \rightarrow P$ be the projection and $\x$, $X \coloneqq \pi(\x)$ a member of the configuration $\hsC$, $\mathscr{C}$ respectively. We are going to define a determinantal contraction for the formal correspondence $\hsC \mathrel{\leftarrow\mkern-14mu\leftarrow}\!\rightarrow \mathscr{C}$ and find a morphism $\pi : \x \rightarrow X$ with each fiber is a point or a projective line in $\bP^n$.

Writing $z = [z_0 : \cdots : z_n] \in \bP^n$ and let $\x \in \hsC$ be defined by global sections
$$
\sum_{i = 0}^n s^i_j(p) z_i = 0
$$
and $t_l(p) = 0$, where $s^i_j \in H^0(P, \sL_j)$ and $t_l \in H^0(P, \sL_l)$ for $1 \leqslant j \leqslant n + 1$, $n + 2 \leqslant l \leqslant m$. Set
$$
\Delta(p) \coloneqq \det(s^i_j(p)),
$$
which is a global section of the line bundle $\bigotimes_{j = 1}^{n + 1} \sL_j$ on $P$. Since $z_i$ cannot all vanish simultaneously, we have $\Delta(p) = 0$ for $(z, p) \in \bP^n \times P$.

Obviously, the $X \coloneqq \pi(\x)$ is defined by global sections
\begin{center}
$\Delta(p) = 0$ and $t_l(p) = 0$
\end{center}
for $n + 2 \leqslant l \leqslant m$ and thus $X$ belongs to the configuration $\mathscr{C}$.

\begin{definition}
Assume that $\hsC$ and $\sC$ are configurations of dimension 3. We say that a formal correspondence $\hsC \mathrel{\leftarrow\mkern-14mu\leftarrow}\!\rightarrow \mathscr{C}$ gives a determinantal contraction if there is a smooth member $\x$ in $\hsC$ such that the morphism $\pi : \x \rightarrow X$ given in the above process is an isomorphism or a small resolution of a normal variety $X \in \mathscr{C}$ with only isolated singularities.
\end{definition}

The proof of the following theorem will follow the idea outlined in \cite{cdls}.  The main tool used in the proof is the Bertini-type theorem introduced in \S \ref{prelsec}.

\begin{theorem} \label{detcon}
Let $\hsC$ and $\mathscr{C}$ be 3-dimensional CICY configuration matrices as above. Then the formal correspondence $\hsC \mathrel{\leftarrow\mkern-14mu\leftarrow}\!\rightarrow \mathscr{C}$ gives a determinantal contraction.
\end{theorem}

\begin{proof}
Let
$$
\hsL_j =
\left\{
\begin{array}{ll}
p^{\ast}\sO_{\bP^n}(1) \otimes \pi^{\ast}\sL_j & \mbox{if } 1 \leqslant j \leqslant n +1, \\
\pi^{\ast}\sL_j  & \mbox{if } n + 2 \leqslant j \leqslant m,
\end{array}
\right.
$$
where $p$ and $\pi$ are the projections from $\bP^n \times P$ onto $\bP^n$ and $P$ respectively. The basic idea of the proof is to find a suitable Zariski open subset in the space of global sections of $\bigoplus_{j =1}^m \hsL_j$ by repeatedly applying Theorem \ref{berti}.

By Proposition \ref{gene2}, there is a Zariski open set $\hU$ in $H^0 \left(\bP^n \times P, \bigoplus_{j =1}^m \hsL_j\right)$ such that $\x = Z(\sigma)$ is a smooth Calabi--Yau $3$-fold for $\sigma \in \hU$. Under the isomorphism $H^0 \left(\bP^n \times P, \bigoplus\nolimits_{j =1}^m \hsL_j\right) \simeq \bigoplus\nolimits_{j =1}^m H^0 \left(\bP^n \times P,  \hsL_j\right)$, we may assume that $\hU = \prod_{j = 1}^{m} U_j$ where $U_j$ is a Zariski open subset of $H^0 \left(\bP^n \times P,  \hsL_j\right)$ (cf. Remark \ref{genermk}).

By Theorem \ref{berti}, for a general morphism $\tau : \bigoplus_1^{n + 1} \sO_P \rightarrow \bigoplus_{j = 1}^{n + 1} \sL_j$ the expected codimension of the degeneracy locus $D_{k}(\tau)$ is $(n + 1 - k)^2$. In particular, the expected codimension of the degeneracy loci $D_{n - 2}(\tau)$ and $D_{n -1}(\tau)$ in $P$ are nine and four. Using Lemma \ref{prozari}, we may assume that there are Zariski open subsets $V_{ij}$ of $H^0(P, \sL_j)$ for $1 \leqslant i, j \leqslant n + 1$ such that sections  $[s^i_{j}] \in \prod_{i, j = 1}^{n +1} V_{ij}$ correspond to morphisms $\tau$, by identifying $\tau$ with a global section of $(\bigoplus_{j = 1}^{n + 1} \sL_j) \otimes (\bigoplus_{j = 1}^{n + 1} \sO_j)^{\vee}$.

Applying  K\"{u}nneth formula, we have for $1 \leqslant j\leqslant n + 1$
\begin{align}\label{kunn}
H^0(\bP^n \times P, \hsL_j) &\simeq H^0(\bP^n, \sO_{\bP^n}(1)) \otimes H^0(P, \sL_j)   \\
                                           &\simeq  \bigoplus\nolimits_{i= 1}^{n + 1} \left(H^0(P, \sL_j) \cdot z_{i - 1}\right) \notag
\end{align}
and for $n + 2 \leqslant j\leqslant m$
$$
H^0(\bP^n \times P, \hsL_j) \simeq H^0(P, \sL_j),
$$
where $\{z_0, \cdots, z_n\}$ is a basis of $H^0(\bP^n, \sO_{\bP^n}(1))$. Therefore we can identify $\sigma \in \hU$ with global sections $\sum_{i = 0}^n s^i_j(p) z_i = 0$ and $t_l(p) = 0$ where $s^i_j \in H^0(P, \sL_j)$ and $t_l \in H^0(P, \sL_l)$ for $1 \leqslant j \leqslant n + 1$ and $n + 2 \leqslant l \leqslant m$.

Using (\ref{kunn}) and Lemma \ref{prozari} again, $\prod\nolimits_{i = 1}^{n + 1} V_{ij}$  can be thought of as a Zariski open subset of $H^0(\bP^n \times P, \hsL_j)$ for $1 \leqslant j \leqslant n + 1$. Replacing $\hU$ with
$$
\left( \prod\nolimits_{j = 1}^{n + 1} ( U_j  \cap ( \prod\nolimits_{i = 1}^{n + 1} V_{ij} ) ) \right) \times \prod\nolimits_{j = n + 2}^m U_j,
$$
we get the desired Zariski open set.

We are now in a position to show the existence of determinantal contractions. Pick a section $\sigma = ( \sum_i s^i_j z_i, t_j ) \in \hU$, we notice that, for $p \in P$,
the dimension of $\pi^{- 1}(p)$ is less than two  if and only if the corank of the matrix $[s^i_j(p)]$ is less than or equal to two, i.e., $\mathrm{rank}[s^i_j(p)] \geqslant n - 1$. From $\dim P = m- n + 3$, the number of sections $t_j$'s is equal to $\dim P - 4$. Set $Y$ be the $4$-dimensional smooth variety $Z(t_{n + 2}, \cdots, t_m)$. Since the expected codimension $D_{n - 2}([s^i_j])$ and $D_{n -1}([s^i_j])$ are nine and four, the intersection of $Y$ with $D_{n - 2}([s^i_j])$ and $D_{n -1}([s^i_j])$ are empty and isolated points respectively.

According to that $X = \pi(\x)$ is defined by $\Delta |_{Y} = \det(s^i_j) |_Y$ on the smooth variety $Y$ and is irreducible, it follows that $X$ is integral and satisfies Serre's $\mathrm{S}_2$ condition \cite[Theorem 14.4 (c)]{fulto}. Since $\x = Z(\sigma)$ is a smooth variety, we have now derived that, for all $\sigma = ( \sum_i s^i_j z_i, t_j ) \in \hU$, the morphism $\pi : \x \rightarrow X$ is a small resolution of the normal variety $X$ with only isolated singularities (which equals $Y \cap D_{n -1}([s^i_j])$). Hence the formal correspondence $\hsC \mathrel{\leftarrow\mkern-14mu\leftarrow}\!\rightarrow \mathscr{C}$ gives a determinantal contraction.
\end{proof}

\begin{remark} \label{cork}
If corank of $[s^i_j(p )]$ is $1$ or $2$ then the solution space of the matrix defines a point or a projective line in $\bP^n$ respectively. Namely, each fiber of $\pi$ is a point or a projective line in $\bP^n$.
\end{remark}

\begin{corollary} \label{cicynumsing}
With notation as in the proof of Theorem \ref{detcon}. Let $\sE = \bigoplus_{i = 1}^{n + 1} \sL_i$ and $\sF = \bigoplus_{i = n + 2}^{m} \sL_i$ be vector bundles  of rank $n + 1$ and $m - n - 1$ on $P$ respectively. For the determinantal contraction $\pi : \x \rightarrow X$, the number of singularities of $X$ is equal to
$$
\int_P \left(c_2(\sE)^2 - c_1(\sE) c_3(\sE)\right) c_{m - n -1}(\sF).
$$
\end{corollary}

\begin{proof}
As in the proof of Theorem \ref{detcon}, the number of singularities of $X$ equals the intersection number $[D_{n -1}([s^i_j])] \cap [Z(t_{n + 2}, \cdots, t_m)] \cap [P]$. By \cite[Theorem 14.4, Example 14.4.1]{fulto}, for the smooth general member $\x$ which is defined by a general section $\sigma = ( \sum_i s^i_j z_i, t_j )$, the fundamental classes $[D_{n -1}([s^i_j])]$ and $[Z(t_{n + 2}, \cdots, t_m)]$ are
\begin{center}
$\left(c_2(\sE)^2 - c_1(\sE) c_3(\sE)\right) \cap [P]$ and $c_{m - n -1}(\sF) \cap [P]$
\end{center}
respectively. This completes the proof.
\end{proof}

\begin{remark}
If $\pi : \x \rightarrow X$ is an isomorphism, that is $\Sing(X) = \varnothing$, the $\hsC \mathrel{\leftarrow\mkern-14mu\leftarrow}\!\rightarrow \mathscr{C}$ is referred to as an ineffective splitting in \cite[p.512]{cdls}. It is easy to see that it is ineffective if and only if $X$ and $\x$ have the same Euler characteristic if and only if the intersection $D_{n -1}([s^i_j]) \cap Z(t_{n + 2}, \cdots, t_m)$ is empty. In the case $n = 1$, the intersection is defined by $\dim P$ sections $s^i_j$ and $t_l$. Therefore the splitting is ineffective if and only if the intersection number
$$
c_2(\sE)^2 \cap [P] = D_1.\cdots. D_{\dim P} = 0
$$
where $D_1, \cdots, D_{\dim P}$ are Cartier divisors defined by $s^0_1, s^1_1, s^0_2, s^1_2$, $t_l$'s respectively.
\end{remark}

We are now ready to prove the connectedness of parameter spaces of CICY configuration matrices.

\begin{theorem}\label{webcicy}
Any two (parameter spaces of) CICY 3-folds in product of projective spaces are connected by a finite sequence of conifold transitions.
\end{theorem}

\begin{proof}
By Proposition \ref{fconn} and Theorem \ref{detcon}, every CICY configuration matrices connect formally and each formal correspondence gives a determinantal contraction $\x \rightarrow X$, which is an isomorphism or a small projective resolution, say $X \in \sC$. According to Corollary \ref{eulchern} and Corollary \ref{cicynumsing}, it follows that $e(\x) - e(\tX) = 2 \left|\Sing(X)\right|$, where $\tX \in \sC$ is a general smooth member. By Proposition \ref{smalltraodp}, the singularities of $X$ are ODPs. Hence each parameter space $[\mathbf{n} \| \mathbf{q}]$ connects to $\mathscr{C}_{1111}$ by conifold transitions (cf. Remark \ref{fconnrmk}).
\end{proof}

\begin{example}[Fiber products of elliptic surfaces]
Consider
$$
\hsC \coloneqq
\left[
\begin{array}{c||cc}
2 & 3 & 0 \\
2 & 0 & 3 \\
1 & 1 & 1
\end{array}
\right]
\mathrel{\leftarrow\mkern-14mu\leftarrow}\!\rightarrow
\mathscr{C} \coloneqq
\left[
\begin{array}{c||c}
2 & 3 \\
2 & 3
\end{array}
\right].
$$
It shall be related to the fiber products of rational elliptic surfaces which was investigated in \cite{scho}.

Let $f_i : S_i \rightarrow \bP^1$ be a relatively minimal, rational, elliptic surface with section for $i = 1, 2$. Then $S_i$ is the 9-fold blowing up of $\bP^2$ at the base points of a cubic pencil which induces the fibration $f_i$ \cite[IV.1.2]{mir}, that is, there are generic homogeneous cubic polynomials $a_i$ and $b_i$ such that $S_i \subseteq \bP^2 \times \bP^1$ is a resolution of indeterminacy of the rational map $C_i : \bP^2 \dashrightarrow \bP^1$ defined by $C_i(x) = [a_i(x) : b_i(x)]$. Obviously, $S_i$ is defined by
$$
P_i(z, x) = z_1 a_i(x) - z_0 b_i(x) = 0
$$
where $[z_0 : z_1] \in \bP^1$ and $x \in \bP^2$.

Let $W = S_1 \times_{\bP^1} S_2$. It is well known that $W$ is a Calabi--Yau 3-fold \cite{scho}.
It is easy to see that $W$ can be obtained as a CICY in $\bP^2 \times \bP^2 \times \bP^1$ defined by $P_1$ and $P_2$. Therefore $W \in \hsC$ and is birational to a member in $\mathscr{C}$ which is defined by the bicubic polynomial $a_0(x) b_1(x) - a_1(x) b_0(x) = 0$.
\end{example}

\begin{example}[Double solids] \label{doubleso}
Consider the CICY configuration matrix
$$
\mathscr{C} \coloneqq
\left[
\begin{array}{c||c}
3 & 4 \\
1 & 2
\end{array}
\right].
$$
Let $x, y$ be a basis of $H^0(\sO_{\bP^1}(1))$. Let $\sL$ be the line bundle of multidegree $(4, 2)$ on $P \coloneqq \bP^3 \times \bP^1$ and $\Gamma = H^0(\sO_{\bP^3}(4))$. By Proposition \ref{gene2}, there is a Zariski open subset of
$$
H^0(P, \sL) \simeq (\Gamma \cdot x^2) \oplus (\Gamma \cdot xy) \oplus (\Gamma \cdot y^2)
$$
such that each section in the open set defines a smooth Calabi--Yau 3-fold.

Choose general quartics $A,B$ and $C$ on $\bP^3$ so that the octic hypersurface $S$ in $\bP^3$ defined by $\Delta  \coloneqq B^2 -4AC \in H^0(\bP^3, \sO_{\bP^3}(8))$ contains $4^3 = 64$ singular points (the three
quartics vanish simultaneously at these $64$ points) and the Calabi--Yau $\x \in \mathscr{C}$ defined by
$$
A x^2 + B x y + C y^2 = 0
$$
is smooth.

Let $X$ be the double cover of $\bP^3$ branched over $S$. We can show that the only singular points of $X$ are ODPs, one for each singular point of $S$. These double covers $X$, called double solids, were firstly studied by Clemens \cite{cle}.

Consider the Stein factorization of the natural projection $\phi$ of $\x$ on $\bP^3$:
$$
\phi = \phi' \circ \pi
$$
where $\phi'$ is finite and $\pi$ has connected fibers. For each $p \in \bP^3$, the fiber $\phi^{- 1}(p)$ consists of $(\mathrm{i})$ two points if $\Delta(p) \neq 0$, $(\mathrm{ii})$ one point if $\Delta(p) = 0$ but at least one of $A(p)$, $B(p)$ and $C(p)$ does not vanish, $(\mathrm{iii})$ a copy of $\bP^1$ if $A(p) = B(p) = C(p) = 0$. Therefore the map $\phi '$ is a double cover of $\bP^3$ (by $(\mathrm{i})$) branched over the octic surface $S$ (by $(\mathrm{ii})$), and the map $\pi : \x \to X$ is a small resolution (by $(\mathrm{iii})$).

For instance, we choose a open set $U$ in $\bP^3$ such that $\sO_{\bP^3}(4)|_U \simeq \sO_U$ and $A|_U$ is nowhere zero. Let $V = \{[x: y] \in \bP^1 \,|\, y \neq 0\}$. On $W \coloneqq U \times V$, we rewrite the equation
$$
A x^2 + B x y + C y^2 = \frac{A}{4}\left[ (2 x + \frac{B}{ A} y)^2 - \frac{\Delta}{A^2} y^2 \right].
$$
Then we get a commutative diagram
\begin{equation*}
\SelectTips{cm}{} \xymatrix{
 \x   & W \cap \x  \ar@{=}[rr]^(.4){\pi} \ar@{_{(}->}[l]  \ar[dr]^{\phi}&  & \Spec_U \frac{\sO_{\bP^3}(U)[T]}{(T^2 - \Delta)} \ar@{^{(}->}[r] \ar[dl]_{\phi'}& X. \\
  & & U & & }
\end{equation*}

Let $\tX$ be a smoothing of $X$, which is a double cover of $\bP^3$ branched over a smooth octic surface $\widetilde{S} \in [3 || 8]$. Then the topological Euler number $e(\tX) = 2 e(\bP^3) - e(\widetilde{S}) = - 296$. Hence, by Proposition \ref{smalltraodp}, the difference of the Euler numbers $e(\x) - e(\tX)$ is $128=2 \cdot 64$ as expected and $X$ is a conifold.
\end{example}

\section{Further discussions on small transitions} \label{remcontr}

We review the definition of (projective) small transitions.

\begin{definition}
Let $\x \rightarrow X$ be a small projective resolution of a Calabi--Yau $3$-fold $X$, which has terminal singularities. If $X$ can be smoothed to a Calabi--Yau manifold $\widetilde{X}$, then
the process of going from $\x$ to $\widetilde{X}$ is called a
small transition and denoted by a diagram $\widehat{X} \rightarrow X \rightsquigarrow \widetilde{X}$. It is called a conifold transition if all singularities of $X$ are ODPs.
\end{definition}

Conifold transitions play a fundamental role in Reid's fantasy \cite[Section 8]{reid3} (cf.~\S \ref{introsec}), which conjectures that all the moduli spaces of smooth Calabi--Yau 3-folds are connected through conifold transitions. As in the previous section the moduli spaces of CICY 3-folds in product of projective spaces are connected to each other by conifold transitions (cf. \cite{gh2} and Theorem \ref{detcon}). A special yet fundamental question arising from Reid's fantasy is the following:

\begin{question} \label{q1}
Let $\widehat{X} \rightarrow X \rightsquigarrow \widetilde{X}$ be a small transition. Is it true $\x$ can be connected to $\widetilde{X}$ by \emph{a conifold transition} (through a different $X$ of course)?
\end{question}

For a Calabi--Yau 3-fold $X$, Namikawa and Steenbrink proved that $X$ can be deformed to a Calabi--Yau 3-fold with at worst ODPs \cite{nast}. In view of this result, it seems that one may possibly answer Question \ref{q1} affirmatively by finding a deformation direction of $\x$ which deforms $\x \to X$ into $\x_1 \to X_1$ with $X_1$ being a Calabi--Yau conifold. Unfortunately, Namikawa produced a counterexample to this in \cite[Remark 2.8]{nam}. We recall it briefly as follows:

Choose a suitable rational elliptic surface $S$ with six singular fibers of type II (i.e., cuspidal rational curves). Let $X = S \times_{\mathbb{P}^1} S$. Then $X$ is a Calabi--Yau 3-fold with six singular points of $cA_2$ type:
$$
x^2 - y^3 = u^2 - v^3,
$$
which admits smoothings to $\widetilde{X} = S_1 \times_{\Bbb P^1} S_2$ with $S_i \to \Bbb P^1$ having disjoint discriminant loci. A small resolution $\pi: \x \to X$ can also be constructed (see below). Namikawa observed that the exceptional loci cannot be deformed to a disjoint union of $(-1, -1)$-curves. The reason is that a singular fiber of type II splits up into at most two singular fibers of type I, and a general fiber of small deformation of a singularity of $X$ which preserves the small resolution has three ODPs.

To search for a modification of Question 1, we need to study Namikawa's construction of the small resolution $\pi$ carefully. Notice that the diagonal $D \cong S$ in $X$ is a smooth Weil divisor which contains the six singular points and is thus not $\bQ$-Cartier\footnote{In fact the divisor class group of a terminal Gorenstein 3-fold is torsion-free \cite[(5.1)]{kawam}. Hence it suffices to show that $D$ is not Cartier. It follows from commutative algebra: If $(A, \fm)$ is a Noetherian local ring, $f \in \fm$ and $A / (f)$ is a regular local ring of dimension $\dim A - 1$, then $A$ is regular. Then we are done since the smooth Weil divisor $D$ contains $\Sing (X)$.}.
On the other hand, by \cite[Example, p.1220]{nam}, there is a nontrivial automorphism $\tau \in \mathrm{Aut}(X)$ such that $D_{\tau} \coloneqq \tau(D)$ has the same properties as $D$. Then $X' \coloneqq \Bl_D X$ has six ODPs and the exceptional locus of $X' \rightarrow X$ consists of six mutually disjoint $\mathbb{P}^1$s, with each of them passing through one of the six ODPs. Now the small resolution can be constructed as the blowing up of $X'$ along the proper transform $\widetilde{D}_\tau$ of $D_{\tau}$, with $\pi$ being composed of morphisms $\x \rightarrow X' \rightarrow X$. It admits exceptional trees, composed of couples of rational curves intersecting at one point.

Now comes the key point. Using Friedman's criterion, $X' \rightarrow X$ can be deformed to a small resolution $Y' \rightarrow Y$ where $Y'$ is smooth and $Y$ has only ODPs. Thus we have decomposed the small transition $\x \to X \rightsquigarrow \widetilde{X}$ into two conifold transitions $\x \to X' \rightsquigarrow Y'$ and $Y' \to Y \rightsquigarrow \widetilde{X}$:
$$
\SelectTips{cm}{}
\xymatrix{\x = \Bl_{\widetilde{D}_\tau} X' \ar[d] & Y' \ar[d] \\
X' = \Bl_D X\ar[d] \ar@{~>}[ur] & Y \ar@{<~>}[dr] \\
 X \ar@{~>}[ur] \ar@{~>}[rr]&& \widetilde{X}.}
$$
Combining the above discussions, we modify Question \ref{q1} as follows:

\begin{question} \label{q2}
Let $\widehat{X} \rightarrow X \rightsquigarrow \widetilde{X}$ be a small transition. Up to deformations of contractions and flops, is it true $\x$ can be connected to $\widetilde{X}$ through \emph{a sequence of conifold transitions}?
\end{question}

\begin{remark} \label{rmkflop}
We prefer to identify Calabi--Yau 3-folds which can be connected by a sequence of flops. The reason is that many invariants are preserved under flops, e.g. quantum invariance \cite{LR, LLW1} (see also \cite{clw} for a survey on recent development), the Kuranishi (miniversal deformation) spaces \cite[(12.6)]{komo}, analytic type of singularities \cite[(4.11)]{kol}, integral cohomology groups, etc. (see \cite[(3.2.2)]{kol2}).
\end{remark}

Before proceeding further, we review here the deformation theory of Calabi--Yau $3$-folds. Fix a Calabi--Yau $3$-fold $X$. Let $\Def (X)$ be the Kuranishi space for flat deformations of $X$ (cf. \cite[(11.3)]{komo} and the references therein). By \cite[Theorem A]{nam2}, the Kuranishi space $\Def (X)$ is smooth\footnote{In this paper, we stick to Calabi--Yau 3-folds $X$ with at worst terminal singularities. If we relax the class of singularities, then $\Def(X)$ might be singular. Indeed, there is a Calabi--Yau 3-fold with canonical singularities whose Kuranishi space is singular \cite{grosobs}.}.

Fix a small projective \emph{partial} resolution $\pi : X' \rightarrow X$, where the Calabi--Yau $3$-fold $X'$ has at worst terminal singularities. Since $X$ has only rational singularities, there is a natural map of germs of smooth complex spaces $\pi_{\ast} : \Def(X') \rightarrow \Def(X)$ (cf. \cite[(11.4)]{komo} or \cite[(1.4)]{wah} on the level of deformation functors). According to that $\pi$ is small, it follows that the natural map $\pi_{\ast}$ is a closed immersion (cf. \cite[(1.12)]{wah}, \cite[(2.3)]{nam}).

Let $\x \to X$ be a small projective resolution. We can show that there is a closed immersion $\Def (\x)$ into $\Def(X')$. Indeed, let $\x'$ be a $\bQ$-factorialization of $X'$ \cite[(4.5)]{kawam}. Since Calabi--Yau 3-folds $\x$ and $\x'$ are connected by a sequence of flops \cite{kawam, kol} and the Kuranishi spaces are unchanged under flops \cite[(12.6)]{komo}, we get
$$
\Def (\x) \simeq \Def (\x') \hookrightarrow \Def(X').
$$
Remark that $\x'$ is also smooth (cf. Remark \ref{rmkflop}, \cite[(4.11)]{kol}).

According to that the Gorenstein terminal 3-fold singularities $p \in X$ are precisely the isolated compound Du Val (cDV for short) hypersurface singularities \cite[(1.1)]{reid2}, it follows that the miniversal deformation space $\Def \left(p \in X\right)$ is smooth (cf. \cite[(4.61)]{komo2} and the references therein). There is a natural restriction homomorphism $\Def(X) \to \Def \left(p \in X\right)$ for every singular point $p \in X$.

To attack Question \ref{q2}, we introduce \emph{primitive} small transitions:

\begin{definition} \label{pritran}
A small transition $\x \xrightarrow{\pi} X \rightsquigarrow \widetilde{X}$ is said to be primitive if it satisfies the following two conditions:
\begin{enumerate}
  \item \label{pritran1} For any small projective partial resolution $X' (\neq X)$ of $X$, the closed immersion $\Def(\x) \hookrightarrow \Def(X')$ of Kuranishi spaces is an isomorphism.

  \item \label{pritran2} The composition
  \begin{equation} \label{locglo}
  \Def (\x) \xhookrightarrow{\pi_{\ast}} \Def(X) \to \prod_{p \in \Sing (X)} \Def \left(p \in X\right)
  \end{equation}
  is trivial.
\end{enumerate}
\end{definition}

\begin{remark}
In \cite{nam2}, Namikawa discusses a natural stratification on the Kuranishi space $\Def (X)$ by means of small projective partial resolutions of $X$. Let $Y_1 \coloneqq \Def (\x)$ and let $Y_0$ be the complement of $Y_1$ in $\Def (X)$. The condition (\ref{pritran1}) in Definition \ref{pritran} means that strata of $\Def (X)$ are only $Y_0$ and $Y_1$. Note that if the relative Picard number of $\pi$ is one, then the condition (\ref{pritran1}) is automatically satisfied.

To explain the condition (\ref{pritran2}) in Definition \ref{pritran}, let $\Pi : \chX \rightarrow \cX$ be a flat family  over the unit disk $\Delta$ in $\bC$ with $\Pi |_{t = 0} = \pi$.
By restricting the deformation $\cX \to \Delta$ of $X$ to a sufficiently small open neighborhood $V_i$ of a singular point $p_i \in X$, we get a deformation $\cV_i \to \Delta$ of the singular point $p_i$. Then the condition (\ref{pritran2}) implies that $\Pi^{- 1} (\cV_i ) \to \cV_i$ is isomorphic to the trivial family $\pi^{- 1} (V_i) \times \Delta  \to V_i \times \Delta$ over the unit disk $\Delta$.
\end{remark}

We recall the Namikawa's criterion for the smoothability \cite[(2.5)]{nam}.

\begin{theorem} [\cite{nam}] \label{namsm}
Let $X$ be a Calabi--Yau $3$-fold. The following two conditions are equivalent:
\begin{enumerate}
\item $X$ is smoothable by a flat deformation;

\item for any small projective partial resolution $X' (\neq X)$ of $X$, $\Def(X') \hookrightarrow \Def(X)$ is not a surjection.
\end{enumerate}
\end{theorem}

As an immediate consequence of Theorem \ref{namsm}, we give an equivalent formulation of the condition (\ref{pritran1}) in Definition \ref{pritran}.

\begin{proposition} \label{main2flop}
Let $\pi : \x \rightarrow X$ be a small projective resolution of a Calabi--Yau 3-fold $X$. Then $X$ satisfies the condition (\ref{pritran1}) in Definition \ref{pritran} if and only if for any small projective partial resolution $X' \rightarrow X$ with $X' \ne X$ the Calabi--Yau 3-fold $X'$ is not smoothable.
\end{proposition}

\begin{proof}
First, we observe that the "only if" implication follows immediately from Theorem \ref{namsm}. To proof the "if" implication, we recall the defect $\sigma (Y)$ of a variety $Y$. It is the rank of $\mathrm{WDiv} (Y) / \mathrm{CDiv} (Y)$, where $\mathrm{WDiv} (Y)$ (resp. $\mathrm{CDiv} (Y)$) is the Abelian group of Weil (resp. Cartier) divisors of $Y$. Then $\sigma(Y) < \infty$ (resp. = 0) if $Y$ has at most rational singularities \cite[(1.1)]{kawam} (resp. if $Y$ is $\bQ$-factorial). Remark that if $Y$ admits a nontrivial small birational morphism then $\sigma(Y) > 0$.

Fix a small projective partial resolution $X' \rightarrow X$ with $X' \ne X$. By Theorem \ref{namsm} and $X'$ is not smoothable, there is a  small projective partial resolution $X'' (\ne X')$ of $X'$ such that $\Def(X'') \hookrightarrow \Def(X')$ is an isomorphism. Clearly $\sigma(X'') < \sigma(X') < \infty$. Since $X''$ is also a small projective partial resolution of $X$, it is not smoothable. Hence we can repeat this process until we reach a $\bQ$-factorial variety $\x''$ with the isomorphism $\Def(\x'') \hookrightarrow \Def(X'')$. According to that $\x''$ is also a $\bQ$-factorialization of $X$, it follows that the composition
$$
\Def (\x) \simeq \Def (\x'') \hookrightarrow \Def(X'') \hookrightarrow \Def(X')
$$
is an isomorphism, which completes the proof of the "if" implication.
\end{proof}

The following proposition explains why we use primitive small transitions as building blocks of general small transitions.

\begin{proposition} \label{dectopri}
Every small transition of Calabi--Yau $3$-folds can be decomposed into primitive small transitions up to deformations and flops.
\end{proposition}

\begin{proof}
Let $\x \xrightarrow{\pi} X \rightsquigarrow \widetilde{X}$ be a small transition of Calabi--Yau $3$-folds. We will prove the proposition using induction on the relative Picard number $\rho \coloneqq \rho(\x / X)$.

Observe that the condition (\ref{pritran1}) in Definition \ref{pritran} is automatically satisfied in the case $\rho = 1$. Suppose that the composition map $(\ref{locglo})$ in Definition \ref{pritran} is not trivial. The key point is just that the Du Val surface singularities have no moduli. In fact, given such a nontrivial deformation $\Pi : \chX \rightarrow \cX$ of $\pi$ over the unit disk $\Delta$ in $\bC$ with $\Pi |_{t = 0} = \pi$, there is a nontrivial holomorphic map $\Delta \to \Def(p_i \in X)$. Since $(p_i \in X)$ is an isolated cDV singularity, it is a $1$-parameter family $f_i$ over a $1$-dimensional disk $\Delta_i$ of Du Val surface singularities (cf. \cite{reid2} or \cite[Section 1]{nam2}). Let $S_i \coloneqq f_i^{- 1}(0)$. By the versality of $\Def (S_i)$ there is a nontrivial homomorphic map $\psi_i : \Delta \times \Delta_i \to \Def(S_i)$.

Let $F_i$ be the miniversal family for the deformation of $S_i$. Since the Milnor number of hypersurface singularities is upper semicontinuous under deformations, the Milnor number of the isolated Du Val surface singularity $S_i$ is greater than or equal to the sum of the Milnor numbers at all singularities of the Du Val surface $F_i^{- 1}(\psi_i (t, w))$ for $(t, w) \in \Delta \times \Delta_i$. Recall that the isolated Du Val surface singularity $S_i$ is simple (and therefore of type $A_n$, $D_n$, $E_6$, $E_7$, $E_8$). Hence, for sufficiently small $t \in \Delta \setminus \{0\}$, the Milnor number at a singularity of $F_i^{- 1}(\psi_i (t, 0))$ is less than the Milnor number of the isolated Du Val surface singularity $S_i= F_i^{- 1}(\psi_i (0, 0))$ (cf. \cite[Theorem 5]{gabr} or \cite[Remark 4.42]{komo2}) and then we replace the original small transition $\x \xrightarrow{\pi} X \rightsquigarrow \widetilde{X}$ with $\chX_t \xrightarrow{\Pi |_t} \cX_t \rightsquigarrow \widetilde{X}$.  Repeating this process finitely many times leads to a primitive small transition.

Now assume that $\rho \geqslant 2$ and that there is a small projective partial resolution $X' \to X$ with $X' \neq X$ such that the Calabi--Yau $3$-fold $X'$ is smoothable. Take a $\bQ$-factorialization $\x'$ of $X'$. Since $\x$ and $\x'$ are both $\bQ$-factorialization of $X$ and $\x$ is smooth, they are connected by flops
$$
\SelectTips{cm}{} \xymatrix{\x \ar@{-->}[r]^{\text{flops}}  \ar@/_1.4pc/[rrr]^{\pi} & \x' \ar[r]& X' \ar[r] & X}
$$
and $\x'$ is also smooth. Then we replace $\x \xrightarrow{\pi} X \rightsquigarrow \widetilde{X}$ with the new small transition $\x' \rightarrow X \rightsquigarrow \widetilde{X}$. By assumption, $X' \to X$ can be deformed to a small projective resolution $Y' \to Y$ where $Y'$ is smooth. Thus we have decomposed $\x' \rightarrow X \rightsquigarrow \widetilde{X}$ into two small transitions $\x' \to X' \rightsquigarrow Y'$ and $Y' \to Y \rightsquigarrow \widetilde{X}$.

By induction, we may assume that, for any small projective partial resolution $X' (\neq X)$ of $X$, the Calabi--Yau 3-fold $X'$ is not smoothable. By Proposition \ref{main2flop}, the small transition $\x \xrightarrow{\pi} X \rightsquigarrow \widetilde{X}$ satisfies the condition (\ref{pritran1}) in Definition \ref{pritran}. Using the same argument as in the case $\rho =1$ yields a primitive small transition, and the proof is completed.
\end{proof}

If we want to approach Question \ref{q2}, understanding primitive small transitions becomes essential. The following theorem provides the first step towards this problem:

\begin{theorem} \label{main2}
Let $\pi : \x \rightarrow X$ be a small projective resolution of a Calabi--Yau 3-fold $X$. If the natural closed immersion $\Def(\widehat{X}) \hookrightarrow \Def(X)$ of Kuranishi spaces is an isomorphism then the singularities of $X$ are ODPs. Moreover, the number of ODPs is equal to the relative Picard number $\rho(\x / X)$.
\end{theorem}

We note that Theorem \ref{main2} is a generalization of \cite[(5.1)]{gros}.

\begin{proof}
The proof is by induction on the relative Picard number $\rho \coloneqq \rho(\x / X)$. Observe that $X$ is not smoothable by Theorem \ref{namsm}. For the case $\rho = 1$, the result follows from the above observation and \cite[(5.1)]{gros}.

To prove the case $\rho \geqslant 2$, we recall some facts about extremal rays. Let $D$ be the pullback of an ample divisor under the morphism $\pi$. By Kodaira's Lemma, a linear system $|mD - A|$ is nonempty for any ample divisor $A$ on $\x$ and $m \gg 0$. Pick a divisor $E \in |mD - A|$, which is relatively antiample by the relative Kleiman's criterion for ampleness. Let $\overline{NE}(\x / X)$ be the relative Mori cone. It is a convex (polyhedral) cone generated by (finitely many) exceptional curves of $\pi$. Using the Cone Theorem \cite[(3.25)]{komo2}, we have a klt pair $(\x, \varepsilon E)$ for $0 < \varepsilon \ll 1$ with $\mathscr{O}_{\x}(- E)$ being $\pi$-ample such that
$$
\overline{NE}(\x / X) = \sum\nolimits_{i = 1}^{k} \mathbb{R}_{\geqslant 0} [C_i],
$$
where $\mathbb{R}_{\geqslant 0} [C_i]$ are different extremal rays and $k \geqslant \rho$. Notice that every face of $\overline{NE}(\x / X)$ is a $(K_{\x} + \varepsilon E)$-negative extremal face. It is also evident that the number of irreducible components of $\Exc(\pi)$ is at least $\rho$.

Suppose that our assertion is valid for small resolutions with the relative Picard number less than $\rho$, and let $\pi: \x \to X$ be a small projective resolution with $\rho(\x/X) = \rho$. We first claim that the number of irreducible components of $\Exc(\pi)$ is the relative Picard number $\rho$.

Let $U = \x \setminus \pi^{- 1}(\sing(X))$. Consider the following long exact sequence
\begin{equation}\label{lexa}
0 \rightarrow H^1(\Omega^2_{\x}) \rightarrow H^1(U, \Omega^2_{U}) \rightarrow \bigoplus_{p \in \sing(X)} H^2_{\pi^{- 1}(p)}(\Omega^2_{\x}) \xrightarrow{\alpha} H^2(\Omega^2_{\x})
\end{equation}
where $H^1_{\pi^{- 1}(p)}(\Omega^2_{\x})$ is vanishing for all $p \in \sing(X)$ by the depth argument (cf. Lemma \ref{diff}). Since $X$ is Calabi-Yau, $\Def(X)$ is smooth \cite{nam2} and the tangent space of $\Def(X)$ is isomorphic to $H^1(U, \Omega^2_U)$, by Schlessinger's result \cite{fri, sch}. According to the assumption of the theorem, the dimension of $\Def(\x)$ and $\Def(X)$ are the same. Then we get $h^1(\Omega^2_{\x}) = h^1(U, \Omega^2_{U})$ and thus $\alpha$ is injective. Since the image of $\alpha$ is just the vector space generated by the fundamental classes of irreducible components of $\pi^{- 1}(\sing(X))$, we get $\mathrm{rank}(\alpha) = \rho$. According to Lemma \ref{diff}, it follows that the dimension of $\bigoplus_{p} H^2_{\pi^{- 1}(p)}(\Omega^2_{\x})$ is greater than or equal to the number of irreducible components of $\pi^{- 1}(\sing(X))$ which is at least $\rho$. Hence we conclude that the number of irreducible components of $\pi^{- 1}(\sing(X))$ is exactly $\rho$.

Notice that now we have
$$
\overline{NE}(\x / X) = \bigoplus\nolimits_{i = 1}^{\rho} \mathbb{R}_{\geqslant 0} [C_i].
$$
If any two curves have non-empty intersection, say $C_1$ and $C_2$, we let $F$ be the cone generated by $[C_1]$ and $[C_2]$. It is indeed a face since there are precisely $\rho$ generators of the $\rho$-dimensional cone $\overline{NE}(\x / X)$. Let $\pi' : \x \rightarrow X'$ be the contraction of the $(K_{\x} + \varepsilon E)$-negative extremal face $F$. By the induction hypothesis, the singularities of $X'$ consist of exactly two ODPs and $\Exc(\pi') = C_1 \coprod C_2$. This contradicts to that $C_1 \cap C_2 \neq \varnothing$, and thus $\Exc(\pi)$ is a disjoint union of irreducible rational curves. By the above argument using the induction hypothesis and the Cone theorem, we infer that the singularities of $X$ are ODPs (the normal bundle of an irreducible exceptional curve in $\x$ is $\sO_{\bP^1}(- 1)^{\oplus 2}$).
\end{proof}

We can use Theorem \ref{main2} to give a necessary condition for primitive small transitions with the relative Picard number $\geqslant 2$.

\begin{corollary}
Let $\widehat{X} \xrightarrow{\pi} X \rightsquigarrow \widetilde{X}$ be a primitive small transition. If the relative Picard number of $\pi$ is great than or equal to two, then for any nontrivial factorization $\widehat{X} \rightarrow X' \rightarrow X$ with $X' \neq X$ the singularities of $X'$ are ODPs. Moreover, the number of ODPs of $X'$ equals $\rho(\x) - \rho(X')$.
\end{corollary}

\begin{question} \label{q3}
Can one classifies primitive small transitions? Or more ambitiously, is it true a primitive small transition is necessarily a conifold transition?
\end{question}

It amounts to studying the \emph{global} deformation theory of the small contraction $\widehat{X} \to X$. Notice that in the case of standard web (CICY inside product of projective spaces, Theorem \ref{webcicy}), we have used a Bertini-type theorem for degeneracy loci to play the role of the required deformation theory. For a general small transition, a deeper analysis of globalizing the local deformations is needed.

\end{document}